\documentclass[11pt, twoside]{amsart}
\usepackage{amssymb,amsthm,amsmath}\usepackage{txfonts}\usepackage{srcltx}
\usepackage[english]{babel}

\newtheorem{teo}{Theorem}[section]
\newtheorem{oss}[teo]{Remark}

\newtheorem{Prop}[teo]{Proposition}
\newtheorem{lemma}[teo]{Lemma}

\newtheorem{Defi}[teo]{Definition}

\newtheorem{corollario}[teo]{Corollary}
\newtheorem{no}[teo]{Notation}

\newcommand{\rr}{_{^{_\mathcal{R}}}}

\newcommand{\cc}{_{^{_\HH}}}

\newcommand{\res}{\mathop{\hbox{\vrule height 7pt width .5pt depth 0pt
\vrule height .5pt width 6pt depth 0pt\,}}\nolimits}

\def \op{^\perp}

\newcommand{\LL}{\mathop{\hbox{\vrule height .5pt width 6pt depth
0pt \vrule height 7pt width .5pt depth 0pt\,}}\nolimits}

\newcommand{\ngr}{_{^{_{\mathrm Gr}}}}

\def \cin{{\mathbf{C}^{\infty}}}

\def\dim {\mathrm{dim}}
\def\dc {d_{CC}}

\def\ss{_{^{_{\HS}}}}
\def\eu {_{^{_{Eu}}}}

\def\xs{{\mathit{x}\ss}}
\def\x{{\mathit{x}\cc}}
\def\xsp{{\mathit{x}\op\ss}}
\def\xp{{\mathit{x}\op\cc}}

\def\g{h\cc}

\def\dg{\textit{grad}\cc}
\def\qq{\textit{grad}\ss}

\def \per  {\sigma^{2n}\cc}
\def \perh {\sigma^{2n}\cc}

\def\SC{{C^{\gg}}}

\def\UU{\mathcal{U}}

\def \nn{\nu\cc}

\def \nt{\nu^t\cc}

\def \XH{\mathfrak{X}\cc}
\def \XX{\mathfrak{X}}

\def \MS{\mathcal{H}\cc}

\def \MST{(\MS)_t}

\def \P{{\mathcal{P}}}

\def \PH{\P\cc}

\def \Om{\Omega}

\def \R{\mathbb{R}}

\def \div{\mathit{div}}

\def \GG{\mathbb{G}}
\def \gg{\mathfrak{g}}

\def\divh{\div\cc}

\def\lh{\mathcal{L}\ss}

\def\lg{\mathcal{D}\ss}

\def\tsc{\nabla^{^{_{\TT{S}}}}}
\def\gs{\nabla^{_{\HS}}}
\def\gc{\nabla^{_{\HH}}}

\def\UU{\mathcal{U}}

\def\UU{\mathcal{U}}

\def \nn{\nu_{_{\!\HH}}}

\def \cont{{\mathbf{C}}}
\def \Om{\Omega}

\def \R{\mathbb{R}}

\def \cji {c_{j\,i}(x)}
\def \C { C(x):=[\cji]_{j,i},\,\, {j=1,\ldots,m \,,\, i=1,\ldots,n}}

\def \X {X=(X_{1}, \ldots, X_{m_1})}

\def \X0 {X_{1}(0)\!=\!\partial_{x_{1}}, \ldots, X_{m_1}(0)\!=\!\partial_{x_{m_1}}}

\def \HG {\HH\GG}
\def \HS {\HH{S}}

\def \TG {\mathit{T}\GG}
\def \HH {\mathit{H}}

\def \TT {\mathit{T}}
\def \TS {\mathit{T}S}

\def \grad{\textit{grad}}
\def \C0H{\mathbf{C}_{0}^{\infty}(U,\HG)}
\def \C00{\mathbf{C}_{0}^{\infty}(U)}
\def \C01{\mathbf{C}_{0}^{1}(U)}

\def \L1{d\,\mathcal{L}^n}

\def \H1{\mathcal{H}_{{\rm cc}}^{1}}

\def \Vol{{{\mathcal V}ol}^{2n+1}}

\def \exp{\textsl{exp\,}}

\def \Om{\Omega}

\def \R{\mathbb{R}}

\def \cji {c_{j\,i}(x)}
\def \C { C(x):=[\cji]_{j,i},\,\, {j=1,\ldots,m \,,\, i=1,\ldots,n}}

\def \GG{\mathbb{G}}
\def \gg{\mathfrak{g}}

\def \X {X=(X_{1}, \ldots, X_{m_1})}

\def \X0 {X_{1}(0)\!=\!\partial_{x_{1}}, \ldots, X_{m_1}(0)\!=\!\partial_{x_{m_1}}}

\def \HG {\mathit{H}}

\def \C0H{\mathbf{C}_{0}^{\infty}(\Om,\HG)}
\def \C00{\mathbf{C}_{0}^{\infty}(\Om)}
\def \C01{\mathbf{C}_{0}^{1}(\Om)}

\def \exp{\textsl{exp\,}}

\def\GG{\mathbb{G}}

\begin{document}

\vskip 3cm
\begin{center}
{ \LARGE \bf An integral formula on the Heisenberg group}

\vskip 1cm {\large Francescopaolo Montefalcone\footnote{F. M. has been partially supported by the Fondazione CaRiPaRo Project ``Nonlinear Partial Differential Equations: models, analysis, and control-theoretic problems".} }

\end{center}

\markboth{Francescopaolo Montefalcone}{An integral formula on the
Heisenberg group}

\section*{Abstract} \small
Let $\mathbb H^n$ denote the $(2n+1)$-dimensional (sub-Riemannian) Heisenberg group. In this note, we shall prove  an integral identity (see Theorem \ref{RELL}) which generalizes a formula obtained in the Seventies  by Reilly, \cite{Reilly}. Some first applications will be given in Section \ref{SECTION4}.
\\{\noindent \scriptsize \sc Key words and phrases:}
{\scriptsize{\textsf {Heisenberg groups; Sub-Riemannian geometry;
hypersurfaces; Reilly's Formula.}}}\\{\scriptsize\sc{\noindent Mathematics Subject
Classification:}}\,{\scriptsize \,58C99, 26B20, 43A80.}

\normalsize

\tableofcontents

\section{ Introduction and statement of the main result}\label{basicsHYP}

In the last years, the sub-Riemannian geometry of Carnot groups has become a rich research field in both Analysis and Geometric Measure Theory; see, for
instance, \cite{balogh}, \cite{CDPT}, \cite{vari}, \cite{DanGarN8,
gar}, \cite{FSSC3} ,  \cite{Monti}, \cite{Mag},
\cite{Monteb}, \cite{RR}, but of course the  list is far from being complete or exhaustive. General overviews of
sub-Riemannian (or Carnot-Charath\'eodory)  geometries are Gromov, \cite{Gr1},
  and Montgomery, \cite{Montgomery}.

In this paper, our ambient space is the so-called {\it Heisenberg
group} $\mathbb{H}^n,\,n\geq 1,$ which can be regarded as $\mathbb{C}^n\times \R$ endowed with a
polynomial group law $\star:\mathbb{H}^n\times \mathbb{H}^n\longrightarrow\mathbb{H}^n$. Its Lie algebra $\mathfrak{h}_n$
identifies with the tangent space  $\TT_0\mathbb{H}^n$ at the
identity $0\in\mathbb{H}^n$.  Later on,  $(z, t)\in\R^{2n+1}$ will denote exponential coordinates of a
generic point $p\in\mathbb H^n$. Take now a left-invariant frame
$\mathcal{F}=\{X_1, Y_1,...,X_n, Y_n, T\}$ for the tangent bundle $\TT\mathbb{H}^n$,
where $X_i(p):=\frac{\partial}{\partial x_i} -
\frac{y_i}{2}\frac{\partial}{\partial t}$,
$Y_i(p):=\frac{\partial}{\partial y_i} +
\frac{x_i}{2}\frac{\partial}{\partial t}$  and
$T(p):=\frac{\partial}{\partial t}$. Denoting by $[\cdot, \cdot]$ the
 Lie bracket of vector fields, one has
$[X_i,Y_i]=T$ for every $i=1,...,n$ and all other commutators
vanish. Hence, $T$ is the {\it center} of $\mathfrak{h}_n$ and
 $\mathfrak{h}_n$ turns out to be nilpotent and stratified of step 2, i.e.
$\mathfrak{h}_n=\HH\oplus \HH_2$ where $\HH:={\rm span}_{\R}\{X_1,
Y_1,...,X_i,Y_i,...,X_n,Y_n\}\subset\TT\mathbb H^n$ is the  {\it horizontal bundle} and
$\HH_2={\rm span}_{\R}\{T\}$ is the $1$-dimensional (vertical)
subbundle of $\TT\mathbb H^n$ associated with the center of $\mathfrak{h}_n$. From now
on, $\mathbb{H}^n$ will be endowed with the (left-invariant)
Riemannian metric $h:=\left\langle\cdot, \cdot\right\rangle$ which makes
$\mathcal{F}$ an orthonormal frame.
\begin{oss}Hereafter, the pair $(\mathbb H^n, h)$ will be thought of as a Riemannian manifold. By  duality  w.r.t. the metric $h$,  we define a basis of left-invariant 1-forms  for the cotangent bundle $\TT^\ast\mathbb H^n$. Therefore, we have $X^\ast_1=dx_1,\,
dy_1=Y_1,...,X^\ast_i=dx_i, Y^\ast_i=dy_i,...,X^\ast_n=dx_n,Y^\ast_n=dy_n$. Furthermore, one has $\theta:=T^\ast=dt +
\frac{1}{2}\sum_{i=1}^n\left(y_i d x_i  -  x_i d y_i\right)$, which is the {\rm contact form} of
$\mathbb{H}^n$.  The {\rm Riemannian left-invariant volume form}
$\sigma^{2n+1}\rr\in\bigwedge^{2n+1}(\TT^\ast\mathbb{H}^{n})$ is
defined by $\sigma^{2n+1}\rr:=\left(\bigwedge_{i=1}^n dx_i \wedge d
y_i\right) \wedge \theta$ and the measure, obtained by integration of
$\sigma^{2n+1}\rr$, turns out to be the {\rm Haar measure} of $\mathbb{H}^{n}$.\end{oss}

The metric $h=\left\langle\cdot, \cdot\right\rangle$
induces a corresponding metric $h\cc$ on $\HH$, which is used  to measure the length of horizontal curves. The
natural distance in sub-Riemannian geometry is the  {\it
Carnot-Carath\'eodory  distance}  $\dc$, defined  by minimizing
the (Riemannian) length of all (piecewise smooth) horizontal curves
joining two different points. This definition makes sense because,
in view of Chow's Theorem, different points can always be joined by
 horizontal curves.

The stratification of  $\mathfrak{h}_n$ is related with the
existence of a 1-parameter group of automorphisms, called {\it
Heisenberg dilations}, defined by $\delta_s (z, t):=(s z, s^2 t)$,
for every $p\equiv(z, t)\in\R^{2n+1}$. The intrinsic dilations play an important role in
this geometry. In this regard, we stress that the integer $Q=2n+2$, which represents the \textquotedblleft homogeneous dimension\textquotedblright\, of $\mathbb H^n$ (w.r.t. Heisenberg dilations), turns out to be the  dimension of $\mathbb{H}^n$ as a
metric space w.r.t. the {\it CC-distance} $\dc$.

Another key notion is that of \it $\HH$-perimeter,\, \rm but since we are dealing with smooth boundaries, we do not adopt the usual variational definition. So let $S\subset\mathbb{H}^n$ be a smooth hypersurface and let $\nu$ the (Riemannian) unit normal along $S$. The {\it $\HH$-perimeter measure} $\perh$ is the $(Q-1)$-homogeneous measure, w.r.t. Heisenberg dilations, given by
$\perh\res S:=|\PH\nu|\,\sigma^{2n}\rr$, where $\PH:\TG\longrightarrow\HH$ is the orthogonal projection operator onto $\HH$ and $\sigma^{2n}\rr$ is the Riemannian measure on  $S$. We recall that the unit $\HH$-normal along $S$ is the normalized projection onto $\HH$ of the  (Riemannian) unit normal $\nu$, i.e. $\nn:=\frac{\PH\nu}{|\PH\nu|}$ and that the so-called \it characteristic set $C_S$ \rm of $S$ is the zero set of the function $|\PH\nu|$; see Section \ref{sez22}.
 The $\HH$-perimeter $\perh$ is in fact the natural measure on hypersurfaces and it turns out to be equivalent, up to a density function called metric factor (see, for instance, \cite{Mag}), to the spherical $(Q-1)$-dimensional Hausdorff measure associated with $\dc$ (or to any other homogeneous distance).

Below we shall prove a  general integral
identity, which generalizes to the sub-Riemannian setting of the
Heisenberg group $\mathbb{H}^n$ a well-known
  formula, proved  by
 Reilly (see \cite{Reilly}) in his  work concerning  Aleksandrov'
 Theorem; for a very nice presentation of the original result we refer the reader to \cite{LI}.

\begin{teo}[Main result]\label{RELL}Let $D\subset\mathbb{H}^n$ and let $S=\partial D$ be a
$\cont^2$-smooth compact (closed)
 hypersurface. Let $\phi:D\longrightarrow\R$ be a smooth solution to  $\left\{\begin{array}{ll}\Delta\cc\phi=\psi \,\,\mbox{on}\,\,D\\\,\,\,\,\,\,\,\phi=\varphi\,\,\,\mbox{on}\,\,S\end{array}\right.
$.\\ Then
\begin{eqnarray*}&& \int_D \left\{\psi^2-\|{\rm
Hess}\cc\phi\|^2\ngr+2\left\langle \dg \left(T\phi\right)
,(\dg\phi)\op\right\rangle\right\}\sigma^{2n+1}\rr\\&=&
\int_S\left(2\frac{\partial\varphi}{\partial\nn}\left(\lh\varphi-\frac{\varpi}{2}\frac{\partial\varphi}{\partial\nn\op}\right)-\MS
\left(\frac{\partial\varphi}{\partial\nn}\right)^2-
S\cc\left(\qq\varphi,\qq\varphi\right) \right\}\perh,\end{eqnarray*}.
\end{teo}

We stress that:\begin{itemize}\item ${\rm
Hess}\cc$  is the horizontal Hessian operator;
\item  the symbol $X\op$ (whenever $X\in\HH$) denotes a linear skew-symmetric map. More precisely, it is defined by setting $X\op:=-C\cc^{2n+1}X$, where $C\cc^{2n+1}\in\mathcal M_{2n\times 2n}(\R)$ is given by formula \eqref{formula1}; see below.
\item  $\dg$ and $\qq$ denote the horizontal gradient and the horizontal tangent gradient, resp.;
\item $\lh$ denotes a 2nd order horizontal tangential  operator, which plays the role of the classical Laplace-Beltrami operator in Riemannian geometry;
\item $\varpi:=\frac{\nu_T}{|\PH\nu|}$, where $\nu_T=\langle\nu, T\rangle$;
\item $\MS$ is the horizontal mean curvature of $S$;
\item $S\cc$ is the symmetric part of the horizontal 2nd fundamental form of $S$.\end{itemize}

 In Section \ref{SECTION4} we shall prove some  direct applications of our main result.

 Another consequence will be discussed in Section \ref{SECTION4.1}. More precisely, we shall obtain the following formula:
$$\int_{-\varepsilon}^{\varepsilon}ds   \int_{S_t}\left(\MST^2-\left\|S\cc^t\right\|^2\ngr+   \frac{3n-1}{2}(\varpi^t)^2\right)\perh =
-\int_{S^+\cup S^-}\MS\,\perh,$$where $\MST,\,S\cc^t$ and
$\varpi^t$ denote, respectively, the horizontal mean curvature,
the symmetric part of the horizontal 2nd fundamental form, and the
(weighted) vertical part of the normal $\nu^t$ of the hypersurface
$S_t=\{x\in\mathbb H^n: f_t(x)=f(x, t)=0\,\,\forall\,
t\in]-\varepsilon, \varepsilon[\}$. More precisely, we are
assuming that there is a foliation of a (small) spatial
neighborhood of the (compact, closed hypersurface) $S:=S_0$ by
means of level sets of a smooth function  $f:\mathbb
H^n\times]-\varepsilon, \varepsilon[\longrightarrow\R$  (say of class
$\cont^3$) such that:
\begin{itemize}
 \item $| \grad\,f_t|\neq 0$ along $S_t$ for every $t\in]-\varepsilon, \varepsilon[$,
\item $|\dg f_t|=1$ at each non characteristic (abbreviated NC) point of $S_t$;
\end{itemize}see,  Corollary \ref{Corollary4.9}.

As a final remark, we have to mention that, unfortunately,  the original arguments of Reilly (or those in Li' survey \cite{LI}) cannot be  adapted to our context and, above all, it seems to be still a difficult   problem to prove a  generalized version of Aleksandrov'
 Theorem in $\mathbb H^n$ for $n>1$; see \cite{RR} for the case $n=1$. 

\section{Preliminaries}

\subsection{Heisenberg group $\mathbb{H}^n$}\label{hngeo}

The {\it Heisenberg group} $(\mathbb{H}^n,\star)$, $n\geq
1$, is a connected, simply connected, nilpotent and stratified Lie
group of step 2 on $\R^{2n+1}$, w.r.t. a polynomial group
law $\star$; see below. The {\it Lie algebra} $\mathfrak{h}_n$ of
$\mathbb{H}^n$ is a $(2n+1)$-dimensional real vector space
henceforth identified with the tangent space  $\TT_0\mathbb{H}^n$ at
the identity $0\in\mathbb{H}^n$. We adopt {\it exponential
coordinates of the 1st kind} in such a way that every point
$p\in\mathbb{H}^n$  can be written out as
$p=\exp(x_1,y_1,...,x_i,y_i,...,x_n,y_n, t)$. The Lie algebra
$\mathfrak{h}_n$ can  be described by means of a frame
${\mathcal{F}}:=\{X_1,Y_1,...,X_i,Y_i,...,X_n,Y_n,T\}$
of left-invariant vector fields for $\TT\mathbb{H}^n$, where
$X_i(p):=\frac{\partial}{\partial x_i} -
\frac{y_i}{2}\frac{\partial}{\partial t},\,
Y_i(p):=\frac{\partial}{\partial y_i} +
\frac{x_i}{2}\frac{\partial}{\partial t},\,\, i=1,...,n,\,
T(p):=\frac{\partial}{\partial t},$ for every $p\in\mathbb{H}^n$.
More precisely, if $[\cdot, \cdot]$ denote Lie brackets, then the only non trivial commuting relations are $[X_i,Y_i]=T$ {for every} $i=1,...,n$. In other words, $T$ is the {\it center} of
$\mathfrak{h}_n$ and
 $\mathfrak{h}_n$ turns out to be a nilpotent and stratified Lie algebra of
step 2, i.e. $\mathfrak{h}_n=\HH\oplus \HH_2$. The first layer $\HH$
is called {\it horizontal} whereas the complementary layer $\HH_2={\rm span}_{\R}\{T\}$
is called {\it vertical}.  A horizontal left-invariant frame for
$\HH$ is given by ${\mathcal{F}}\cc=\{X_1,
Y_1,...,X_i,Y_i,...,X_n,Y_n\}.$ The group law $\star$ on
$\mathbb{H}^n$ is determined by a corresponding operation
$\diamond$ on  $\mathfrak{h}_n$, i.e. $\exp X\star\exp
Y=\exp(X\diamond Y)$ for every $X,\,Y \in \mathfrak{h}_n,$ where
$\diamond:\mathfrak{h}_n\times \mathfrak{h}_n\longrightarrow
\mathfrak{h}_n$ is defined by $X\diamond Y= X + Y+
\frac{1}{2}[X,Y]$. Thus, for every
 $p=\exp(x_1,y_1,...,x_n,y_n,
t),\,\,p'=\exp(x'_1,y'_1,...,x'_n,y'_n, t')\in \mathbb{H}^n$ we
have
\[p\star p':= \exp\left(x_1+x_1', y_1+y_1',...,x_n+x_n', y_n+y_n', t+t'+ \frac{1}{2}\sum_{i=1}^n
\left(x_i y'_{i}- x'_{i} y_i\right)\right).\] The
inverse of any ${p}\in\mathbb{H}^n$  is given by
${p}^{-1}:=\exp(-{x}_1,-y_1...,-{x}_{n}, -y_n, -t)$ and
$0=\exp(0_{\R^{2n+1}})$. Later on, we shall set $z:=(x_1, y_1,..., x_n, y_n)\in\R^{2n}$ and identify each point $p\in\mathbb H^n$ with its exponential coordinates $(z, t)\in\R^{2n+1}$.

\begin{Defi}\label{dccar} We call {\rm sub-Riemannian metric} $\g$ any
symmetric positive bilinear form on  $\HH$.
The {\rm {CC}-distance} $\dc(p,p')$ between $p, p'\in
\mathbb{H}^n$ is defined by
$$\dc(p, p'):=\inf \int\sqrt{\g(\dot{\gamma},\dot{\gamma})} dt,$$
where the $\inf$ is taken over all piecewise-smooth horizontal
curves $\gamma$ joining $p$ to $p'$. We shall equip
$\TT\mathbb{H}^n$ with the left-invariant Riemannian metric
  $h:=\left\langle\cdot,\cdot\right\rangle$ making ${\mathcal{F}}$ an
orthonormal -abbreviated o.n.- frame and assume $\g:=h|_\HH.$
\end{Defi}

By Chow's Theorem it turns out that every couple of points can
be connected by a horizontal curve, not necessarily unique, and for this reason $\dc$ turns out to be a true metric on $\mathbb{H}^n$ whose topology is equivalent to the standard (Euclidean)
topology of $\R^{2n+1}$; see \cite{Gr1}, \cite{Montgomery}. The
so-called {\it structural constants} (see \cite{Helgason} or \cite{Monte, Monteb}) of $\mathfrak{h}_n$ are
described by the skew-symmetric $(2n\times 2n)$-matrix
\begin{equation}\label{formula1}
C\cc^{2n+1}:=\left[
\begin{array}{ccccc}
  0 & 1 &   \cdot &  0 &  0 \\
  -1 & 0 &    \cdot &  0 &  0 \\
    \cdot &   \cdot & \cdot & \cdot & \cdot \\
  0 & 0 &  \cdot & 0 & 1 \\
  0 & 0 &   \cdot & -1 & 0
\end{array}%
\right].
\end{equation}This  matrix is associated
with the real valued skew-symmetric bilinear map
$\Gamma\cc:\HH\times\HH\longrightarrow \R$ given by
$\Gamma\cc(X, Y)=\left\langle[X, Y], T\right\rangle$.

\begin{no}We shall set \begin{itemize}
 \item $z^\perp:=-C^{2n+1}\cc z=(-y_1, x_1, ...,-y_n,
x_n)\in\R^{2n}\quad\forall\,\,z\in\R^{2n}$;\item  $X^\perp:=-C^{2n+1}\cc X\quad\forall\,\,X\in\HH$.
\end{itemize}
\end{no}
Given  $p\in\mathbb{H}^n$, we shall denote by
$L_p:\mathbb{H}^n\longrightarrow\mathbb{H}^n$ the {\it left
translation by $p$}, i.e. $L_pp'=p\star p'$, for every
$p'\in\mathbb{H}^n$. The map $L_p$ is a group homomorphism and its
differential
${L_p}_\ast:\TT_0\mathbb{H}^n\longrightarrow\TT_p\mathbb{H}^n$,
 ${L_p}_\ast=\frac{\partial (p\star p')}{\partial p'}\left|_{p'=0}\right.$, is given by
${L_p}_\ast={\rm{col}}[X_1(p),Y_1(p),...X_n(p),Y_n(p),T(p)]$.\\
\indent There exists a 1-parameter group of automorphisms
$\delta_s:\mathbb{H}^n \longrightarrow\mathbb{H}^n\,(s\geq 0)$,
called {\it Heisenberg dilations}, defined by $\delta_s p
:=\exp\left(s z, s^2 t\right)$ for every $s\geq 0$, where $p=\exp(z,
t)\in\mathbb{H}^n$. We recall that the {\it homogeneous dimension}
of $\mathbb{H}^n$ is the integer $Q:=2n+2$. By a well-known result of Mitchell (see, for instance, \cite{Montgomery}), this number coincides with
the {\it Hausdorff dimension} of $\mathbb{H}^n$ as metric space
w.r.t. the CC-distance $\dc$; see \cite{Gr1}, \cite{Montgomery}.\\
\indent We shall denote by $\nabla$ the unique {\it left-invariant
Levi-Civita connection} on $\TT\mathbb{H}^n$ associated with the metric
$h=\left\langle\cdot,\cdot\right\rangle$. We observe that, for every $X, Y, Z\in
\XX:=\cin(\mathbb{H}^n, \TT\mathbb{H}^n)$ one has
\[\left\langle\nabla_XY,Z\right\rangle=\frac{1}{2}
\left(\left\langle[X, Y], Z\right\rangle-\left\langle[Y, Z], X\right\rangle + \left\langle[Z,
X], Y\right\rangle\right).\]\indent For every $X,
Y\in\XX\cc:=\cin(\mathbb{H}^n,\HH)$, we shall set $\gc_X
Y:=\PH\left(\nabla_X Y\right),$ where  $\PH$ denotes  orthogonal
projection  onto $\HH$. The operation $\gc$ is a vector-bundle connection later called
{\it $\HH$-connection}; see \cite{Monteb} and references therein.
It is not difficult to see that $\gc$ is {\it flat}, {\it
compatible with the sub-Riemannian metric} $h\cc$ and {\it
torsion-free}. These properties follow from the very
definition of $\gc$ and from the corresponding properties of the
Levi-Civita connection $\nabla$.

\begin{Defi}
For any $\psi\in\cin(\mathbb{H}^n)$, the {\rm $\HH$-gradient of
$\psi$} is the horizontal vector field $\dg \psi\in\XX\cc$
such that $\left\langle\dg \psi,X \right\rangle= d \psi (X) = X \psi$ for
every $X\in \XH$. The {\rm $\HH$-divergence} $\divh X$ of
$X\in\XX\cc$ is defined, at each point $p\in \mathbb{H}^n$, by
$\divh X(p):= \mathrm{Trace}\left(Y\longrightarrow \gc_{Y} X
\right)(p)\,\,(Y\in \HH_p).$ The {\rm $\HH$-Laplacian}
$\Delta\cc$  is the 2nd order differential operator given by
$\Delta\cc\psi := \div\cc(\dg\psi)$ for every $\psi\in
\cin(\mathbb{H}^n)$.
\end{Defi}

Having fixed a left-invariant Riemannian metric $h$ on
$\TT\mathbb{H}^n$, one  defines by duality (w.r.t. the left-invariant metric $h$) a global co-frame
${\mathcal{F}}^\ast:=\{X^\ast_1,Y^\ast_1,...,X^\ast_i,Y^\ast_i,...,X^\ast_n,Y^\ast_n,T^\ast\}$
of  left-invariant $1$-forms for the cotangent bundle
$\TT^\ast\mathbb{H}^n$, where $X_i^\ast= dx_i,\,
Y_i^\ast=dy_i\,\,(i=1,...,n)$, $\theta:=T^\ast=dt +
\frac{1}{2}\sum_{i=1}^n\left(y_i d x_i  -  x_i d y_i\right).$ The
differential $1$-form $\theta$ represents the {\it contact form} of
$\mathbb{H}^n$.   The {\it Riemannian left-invariant volume form}
$\sigma^{2n+1}\rr\in\bigwedge^{2n+1}(\TT^\ast\mathbb{H}^{n})$ is
given by $\sigma^{2n+1}\rr:=\left(\bigwedge_{i=1}^n dx_i \wedge d
y_i\right) \wedge \theta$ and the measure obtained by integrating
$\sigma^{2n+1}\rr$ is the {\it Haar measure} of $\mathbb{H}^{n}$.

\subsection{Hypersurfaces}\label{sez22}

 Let $S\subset\mathbb{H}^n$ be a
hypersurface of class $\cont^r$ $(r\geq 1)$ and let $\nu$ be the (Riemannian)
unit normal along $S$. We recall that the Riemannian measure
$\sigma^{2n}\rr\in\bigwedge^{2n}(\TT^\ast S)$ on $S$ can
be defined by {\it contraction}\footnote{Let $M$ be a Riemannian
manifold. The linear map $\LL: \Lambda^r(\TT^\ast
M)\rightarrow\Lambda^{r-1}(\TT^\ast M)$ is defined, for $X\in \TT
M$ and $\omega^r\in\Lambda^r(\TT^\ast M)$, by $(X \LL \omega^r)
(Y_1,...,Y_{r-1}):=\omega^r (X,Y_1,...,Y_{r-1})$; see, for
instance, \cite{FE}. This operation is called {\it contraction} or
{\it interior product}.} of the volume form
$\sigma^{2n+1}\rr$ with the unit normal $\nu$ along $S$, i.e.
$\sigma^{2n}\rr\res S := (\nu\LL \sigma^{2n+1}\rr)|_S$.

We say that $p\in S$ is a {\it characteristic point}  if
$\dim\,\HH_p = \dim (\HH_p \cap \TT_p S)$. The {\it characteristic
set} of $S$ is the set of all characteristic points, i.e. $
C_S:=\{x\in S : \dim\,\HH_p = \dim (\HH_p \cap \TT_p S)\}.$ It is
worth noticing that $p\in C_S$ if,
 and only if, $|\PH\nu(p)|=0$. Since $|\PH\nu(p)|$ is continuous
 along $S$, it follows that
$C_S$ is a closed subset of $S$, in the relative topology. We
stress that the
$(Q-1)$-dimensional Hausdorff measure of $C_S$ vanishes, i.e.
$\mathcal{H}_{CC}^{Q-1}(C_S)=0$; see \cite{balogh}, \cite{Mag}.
\begin{oss}\label{CSET}Let $S\subset \mathbb{H}^n$ be a $\cont^2$-smooth hypersurface.
By using {\rm Frobenius' Theorem}
 about integrable distributions, it can be shown that the topological dimension
 of $C_S$ is strictly less than $(n+1)$; see also \cite{Gr1}. For deeper results  about the size of $C_S$
  in $\mathbb{H}^n$, see \cite{balogh}, \cite{Bal3}.\end{oss}
Throughout this paper we shall make use of a (smooth) homogeneous measure on
hypersurfaces, called {\it $\HH$-perimeter measure}; see also
\cite{FSSC3}, \cite{G}, \cite{DanGarN8, gar}, \cite{Mag},
\cite{Monte, Monteb}, \cite{P1},  \cite{RR}.

\begin{Defi}[$\perh$-measure]\label{sh}
 Let $S\subset\mathbb{H}^n$ be a $\mathbf{C}^1$-smooth
non-characteristic (henceforth abbreviated as NC)
 hypersurface and let $\nu$ be the unit
normal vector along $S$. The  {\rm unit $\HH$-normal} along $S$ is defined by $\nn:
=\frac{\PH\nu}{|\PH\nu|}.$ Then, the {\rm $\HH$-perimeter form}
$\perh\in\bigwedge^{2n}(\TT^\ast S)$ is the contraction of the
volume form $\sigma^{2n+1}\rr$ of $\mathbb{H}^n$ by the
horizontal unit normal $\nn$, i.e. $\perh \res S:=\left(\nn \LL
\sigma^{2n+1}\rr\right)\big|_S.$\end{Defi} If
$C_S\neq\emptyset$ we  extend $\perh$ up to $C_S$  by
setting $\perh\res C_{S}= 0$. It turns out that $\perh \res S =
|\PH \nu |\,\sigma^{2n}\rr\, \res S$.

Moreover, at each $p\in S\setminus {C}_S$ one has $\HH_p= {\rm
span}_\R\{\nn(p)\} \oplus \mathit{H}_p S$, where  $\mathit{H}_p
S:=\HH_p\cap\TT_p S$. This allow us to define, in the obvious way, the
associated subbundles $\HS \subset \TS$ and $\nn S$  called {\it horizontal tangent bundle} and {\it horizontal
normal bundle} along $S\setminus C_S$,
respectively.  On the other hand, at each  characteristic point $p\in C_S$,  only  $\HS$ is well-defined and we have  $\HH_pS=\HH_p$ for any $p\in C_S$.
\begin{Defi}
 If $\UU\subseteq S$ is an open set, we  denote by $\cont^i\ss(\UU),\,(i=1, 2)$ the space of functions whose
$\HS$-derivatives up to the $i$-th order are continuous on $\UU$. We  denote by $\cont^i\ss(\UU, \HS),\,(i=1, 2)$ the space of functions with target in $\HS$, whose
$\HS$-derivatives up to $i$-th order are continuous on $\UU$.
\end{Defi}

 Another important geometric object  is given by
$\varpi:=\frac{\nu_{T}}{|\PH\nu|}$; see \cite{Monte, Monteb},
\cite{gar}. Although the function $\varpi$ is not defined at $C_S$, we have $\varpi\in L^1_{loc}(S,
\perh)$.

The following definitions can be found in \cite{Monteb}, for
general Carnot groups. {\it Below, unless otherwise specified, we
shall assume that $S\subset\mathbb{H}^n$ is a $\cont^2$-smooth
NC hypersurface} (i.e. non-characteristic). Let $\tsc$ be the connection on
$S$ induced from the Levi-Civita connection $\nabla$ on
$\mathbb{H}^n$. As for the horizontal connection $\gc$, we define
a \textquotedblleft partial connection\textquotedblright\, $\gs$
associated with the subbundle $\HS\subset\TT S$ by setting
$$\gs_XY:=\P\ss\left(\tsc_XY\right)$$ for every
$X,Y\in\XX^1\ss:=\cont^1(S, \HS)$, where
$\P\ss:\TT{S}\longrightarrow\HS$ denotes the orthogonal projection
operator of $\TT{S}$ onto $\HS$. Starting from the orthogonal
splitting
 $\HH=\nn S\oplus\HS$,
it can be shown that$$\gs_XY=\gc_X Y-\left\langle\gc_X
Y,\nn\right\rangle \nn\quad\mbox{for every}\,\,X,
Y\in\XX^1\ss.$$\begin{Defi}\label{curvmed}If $\psi\in \cont\ss^1(S)$, we define the
$\HS$-{\rm gradient of $\psi$} to be the horizontal tangent
vector field  $\qq\psi\in\XX^0\ss:=\cont(S, \HS)$ such that $\left\langle\qq\psi,X
\right\rangle= d \psi (X) = X \psi$ for every $X\in \HS$. If  $X\in\XX^1\ss$, the {\rm
$\HS$-divergence} $\div\ss X$ of $X$ is given, at each
point $p\in S$, by
$\div\ss X (p) := \mathrm{Trace}\left(Y\longrightarrow
\gs_Y X \right)(p)$ $(Y\in \HH_pS).$ Note that $\div\ss X \in\cont(S)$. The {\rm $\HS$-Laplacian}
$\Delta_{_{\HS}}:\cont\ss^2(S)\longrightarrow\cont(S)$ is the 2nd order differential operator given by
$\Delta\ss\psi := \div\ss(\qq\psi)$ for every $\psi\in
\cont\ss^2({S})$.
The horizontal 2nd fundamental form of $S$ is the bilinear map  ${B\cc}:
\XX^1\ss\times\XX^1\ss \longrightarrow C(S)$ defined by
\[B\cc(X,Y):=\left\langle\gc_X Y,\nn\right\rangle\qquad \mbox{for every}\,\,
X,\,Y\in\XX^1\ss.\]The {\rm horizontal mean curvature} is  the trace
of $B\cc$, i.e. $\MS:=\mathrm{Tr}{B\cc}$.
\end{Defi}

Unless $n=1$,  $B\cc$ is  {\it not symmetric}; see \cite{Monteb}.
Therefore, it is convenient to represent $B\cc$ as a
sum of two operators, one symmetric and the other skew-symmetric,
i.e. $B\cc= S\cc + A\cc$. It turns out
that
$A\cc=\frac{1}{2}\varpi\,C^{2n+1}\cc\big|_{\HS}$; see
 \cite{Monteb}. The linear operator
$C^{2n+1}\cc$ only  acts on horizontal tangent vectors and hence  we shall set
$C^{2n+1}\ss:=C^{2n+1}\cc\left|_{\HS}\right.$. We have  the identity $\|A\cc\|^2\ngr
 =\frac{n-1}{2}\,\varpi^2$; see \cite{Monteb}, Example 4.11, p.
 470. This  can  be proved by means of
 an adapted o.n. frame  ${\mathcal F}$ along $S$.
 Furthermore, we observe that  $\nn^\perp\in {\rm Ker} A\cc$, where
$\nn^\perp=-C^{2n+1}\cc\nn$.

\begin{Defi}\label{movadafr}Let $S\subset\mathbb{H}^n$ be a $\cont^2$-smooth NC hypersurface.
We  call {\rm adapted  frame} along $S$  any o.n. frame
${\mathcal{F}}:=\{\tau_1,...,\tau_{2n+1}\}$ for
$\TT\mathbb{H}^n$
 such that:
\[\tau_1|_S=\nn,\qquad \HH_pS=\mathrm{span}_\R\{\tau_2(p),...,\tau_{2n}(p)\}\quad\mbox{for every}\,\,\,p\in S,
\qquad\tau_{2n+1}:= T.\]Furthermore, we shall set
$I\cc:=\{1,2,3,...,2n\}$ and
 $I\ss:=\{2,3,...,2n\}$.\end{Defi}

\begin{lemma}[see \cite{Monteb}, Lemma 3.8]\label{Sple} Let $S\subset\mathbb{H}^n$  be a $\cont^2$-smooth
NC hypersurface and fix $p\in S$.
We can always choose an adapted o.n. frame
${\mathcal{F}}=\{\tau_1,...,\tau_{2n+1}\}$ along $S$
such that $\left\langle\nabla_{X}{\tau_i}, \tau_j\right\rangle=0$ at $p$ for
every $i, j\in I\ss$ and every $X\in\HH_{p}S$.
\end{lemma}

\begin{lemma}\label{ljjjkl} Let $S\subset\mathbb{H}^n$ be a $\cont^2$-smooth
NC hypersurface. Then
\begin{eqnarray}\label{ljjjkltay}\Delta\ss\phi=\Delta\cc\phi +
\MS\frac{\partial\phi}{\partial\nn}-\left\langle {\rm
Hess}\cc \phi\, \nn, \nn\right\rangle\qquad\forall\,\,\phi\in\cin(\mathbb H^n).\end{eqnarray}
\end{lemma}
\begin{proof}Using an adapted frame ${\mathcal{F}}$, we compute
 \begin{eqnarray*}\Delta\cc\phi&=&\sum_{i\in I\cc}\left(\tau^{(2)}_i-\gc_{\tau_i}\tau_i\right)(\phi)
 \\&=&\tau^{(2)}_1
 (\phi)-\left(\gc_{\tau_1}\tau_1\right)(\phi)+\sum_{i\in I\ss}\left(\left(\tau^{(2)}_i-
 \gs_{\tau_i}\tau_i\right)(\phi)-\left\langle\gc_{\tau_i}\tau_i,
 \tau_1\right\rangle\tau_1(\phi)\right)\\&=&\tau^{(2)}_1
 (\phi)-\left(\gc_{\tau_1}\tau_1\right)(\phi)+ \Delta\ss\phi
 - \MS\tau_1(\phi).
 \end{eqnarray*}Note that the first   identity comes from the usual invariant definition of the Laplace
 operator on Riemannian manifolds (or vector bundles); see \cite{Hicks}.
 Now we claim that \[\tau^{(2)}_1
 (\phi)-\left(\gc_{\tau_1}\tau_1\right)(\phi)=\left\langle{\rm
Hess}\cc \phi\,\tau_1, \tau_1\right\rangle.\]Assuming $\tau_1=\sum_{i\in
I\cc}A^1_iX_i$  yields
\begin{eqnarray*}\tau^{(2)}_1(\phi)=\sum_{i\in
I\cc}\tau_1(A^1_iX_i(\phi))=\sum_{i,j\in
I\cc}\left(\tau_1(A^1_i)X_i(\phi) +
A^1_iA^1_jX_j(X_i(\phi))\right).\end{eqnarray*}Since
$\gc_{\tau_1}\tau_1=\sum_{i, j\in I\cc}\left(\tau_1(A^1_i)X_i+
A^1_iA^1_j{\gc_{X_i}X_j}\right)$ and $\gc_{X_i}X_j=0$, the claim
follows because
\begin{eqnarray}\label{idder}\tau^{(2)}_1
 (\phi)-\left(\gc_{\tau_1}\tau_1\right)(\phi)=\sum_{i,
j\in I\cc} A^1_iA^1_jX_j(X_i(\phi))=\langle {\rm Hess}\cc
\phi\,\tau_1,\tau_1\rangle.\end{eqnarray}
 \end{proof}

\begin{Defi}[Horizontal tangential operators]\label{Deflh} Let $S\subset \mathbb H^n$ be a NC hypersurface. We shall denote by $\lg:\XX^1\ss\longrightarrow\cont(S)$ be
the 1st order
differential operator given by
\begin{eqnarray*}\lg(X):=\div\ss X + \varpi\left\langle C^{2n+1}\cc\nn,
X\right\rangle=\div\ss X -\varpi\left\langle\nn^{\perp}, X\right\rangle\qquad
\mbox{for every}\,\,X\in\XX\ss^1(\HS):=\cont^1\ss(S,\HS).\end{eqnarray*}Moreover, we shall denote by $\lh:\cont\ss^2(S)\longrightarrow\cont(S)$  the 2nd order differential operator given by
\begin{eqnarray*}\lh\varphi:=\lg(\qq \varphi)=\Delta\ss\varphi -\varpi\frac{\partial\varphi}{\partial\nn^\perp}
\qquad\mbox{for every}\,\,\varphi\in\cont^2\ss(S).\end{eqnarray*}
\end{Defi}

 Note that  $\lg(\varphi X)=\varphi\lg X +
\langle\qq \varphi, X\rangle\,\,\forall\,\,X\in\XX\ss^1(\HS),\,\forall\,\,\varphi\in\cont\ss^1(S)$.  These definitions are motivated by
Theorem 3.17 in \cite{Monteb}, which was proved first for NC hypersurfaces with boundary; see \cite{Monte, Monteb}. Actually, it holds true even in case of non-empty
characteristic sets. A
 simple way to formulate this claim is based on the next:
\begin{Defi}\label{adm}
Let $X\in\cont\ss^1(S\setminus C_S, \HS)$ and set
 $\alpha_X:=(X\LL \per)|_S$. We say that $X$ is \rm admissible (for the horizontal divergence formula)  \it if the differential forms $\alpha_X$ and $d\alpha_X$ are continuous on all of $S$. We say that $\phi\in\cont^2\ss(S\setminus C_S)$  is \rm admissible \it if $\qq \phi$ is admissible for the horizontal divergence formula.
\end{Defi}
 If the differential forms $\alpha_X$ and $d\alpha_X$ are continuous on all of $S$,
then Stokes formula holds true; see, for instance, \cite{Taylor}. In particular, we
 stress that: (i) if $X\in\cont\ss^1(S, \HS)$, then $X$ is admissible; (ii) if $\phi\in\cont^2\ss(S)$, then $\phi$ is admissible. The following holds:
\begin{teo}\label{GD2}Let $S$ be a
compact  hypersurface of class $\cont^2$ without
 boundary. Then
\begin{eqnarray}\label{hparts} \int_{S}\lg X\,\per=-\int_{S}\MS\langle X, \nn\rangle\,\per \qquad \forall\,\,X\in\XX^1\cc.\end{eqnarray}\end{teo}

Note that, if $X\in\XX^1\ss$  the first integral on the right hand side vanishes and, in this case,
the formula is referred as \textquotedblleft horizontal divergence formula\textquotedblright.\\

 Finally, we state some useful Green's formulas:
\begin{itemize} \item[{\rm(i)}]
$\int_{S}\lh \varphi\,\perh=0$ for every
$\varphi\in\cont^2\ss(S)$; \item[
{\rm(ii)}]$\int_{S}\psi\,\lh
\varphi\,\perh=-\int_{S}\left\langle\qq\varphi,\qq\psi\right\rangle\,\perh $
 for every
$\varphi,\,\psi\in\cont^2\ss(S)$;\item[{\rm(iii)}]$\int_{S}\lh
 \left(\frac{\varphi^2}{2}\right) \,\perh= \int_{S}\varphi\lh
\varphi\,\perh+ \int_{S}|\qq\varphi|^2\,\perh =0$
for every $\varphi\in\cont^2\ss(S)$.
\end{itemize}

\section{Proof of Theorem \ref{RELL}}

\begin{proof}Below we shall make use of the fixed left-invariant frame
${\mathcal{F}}=\{X_1,X_2,...,X_{2n}, X_{2n+1}\}$, where we have set $X_{2i}:=Y_i$ for every $i=1,...,n$ and $X_{2n+1}:=T$. First, we
compute
\begin{eqnarray*}\frac{1}{2}\Delta\cc|\dg\phi|^2&=&\frac{1}{2}\sum_{i,j\in I\cc}X_iX_i(X_j\phi)^2\\&=&
\sum_{i,j\in I\cc}X_i\left(X_j\phi X_iX_j\phi\right)\\&=&\sum_{i,j\in
I\cc}\left((X_iX_j\phi)^2+
X_iX_i(X_j\phi)X_j\phi\right)\\&=&\sum_{i,j\in
I\cc}\left(\phi^2_{ij} +\Delta\cc(X_j\phi)X_j\phi\right).
\end{eqnarray*}Moreover, we have
\begin{eqnarray*}\Delta\cc(X_j\phi)&=&\sum_{i\in I\cc}
X_iX_i(X_j\phi)\\&=&\sum_{i\in I\cc} X_i\left(X_jX_i\phi+
[X_i,X_j](\phi)\right)\\&=&\sum_{i\in I\cc}
X_jX_i(X_i\phi)+[X_i,X_j](X_i\phi)+X_i\left([X_i,X_j](\phi)\right)\\&=&\sum_{i\in
I\cc}X_j\Delta\cc\phi+
\SC^{2n+1}_{ij}TX_i\phi+\SC^{2n+1}_{ij}X_i(T\phi)\\&=&\sum_{i\in
I\cc}X_j\Delta\cc\phi+
\SC^{2n+1}_{ij}X_i(T\phi)+\SC^{2n+1}_{ij}X_i(T\phi)\\&=&X_j\Delta\cc\phi-2\left\langle
C\cc^{2n+1}\dg (T\phi),X_j\right\rangle.
\end{eqnarray*} From these computations, we infer the formula
\begin{eqnarray}\nonumber\frac{1}{2}\Delta\cc|\dg\phi|^2&=&\|{\rm Hess}\cc\phi\|^2\ngr+
\left\langle\dg\underbrace{(\Delta\cc\phi)}_{=\psi},\dg\phi\right\rangle-2\left\langle
C\cc^{2n+1}\dg
(T\phi),\dg\phi\right\rangle\\\label{derossi}&=&\|{\rm
Hess}\cc\phi\|^2\ngr+
\left\langle\dg\psi,\dg\phi\right\rangle+2\left\langle T( \dg
\phi),C\cc^{2n+1}\dg\phi\right\rangle.\end{eqnarray}{\it
Hereafter we will use the hypothesis $\phi|_S=\varphi$.} By applying the usual
Divergence Theorem, we have
$$\int_D\left\{\psi\Delta\cc\phi+\left\langle\dg\psi,\dg\phi\right\rangle\right\}\sigma^{2n+1}\rr=
\int_S\psi\frac{\partial\varphi}{\partial\nn}\perh$$and we
 get
that$$\int_D\left\langle\dg\psi,\dg\phi\right\rangle\sigma^{2n+1}\rr=-\int_D\psi^2\sigma^{2n+1}\rr+
\int_S\psi\frac{\partial\varphi}{\partial\nn}\perh.$$ Set now
$\chi:=\frac{|\dg\phi|^2}{2}$. By integrating \eqref{derossi}
along $D$ and using the last identity, we obtain

\begin{eqnarray}\label{aquilani}\int_D\Delta\cc\chi\sigma^{2n+1}\rr
&=&\int_D \left(\|{\rm Hess}\cc\phi\|^2\ngr-\psi^2+2\left\langle
T\dg\phi
,C\cc^{2n+1}\dg\phi\right\rangle\right)\sigma^{2n+1}\rr+\int_S\psi\frac{\partial\varphi}{\partial\nn}\perh.\end{eqnarray}
So let  $\mathcal F\cc=\{\tau_1(=\nn),
\tau_2,...,\tau_{2n}\}$ be a horizontal frame for $\HH$ adapted to $S$ and let us compute
\begin{eqnarray}\nonumber\int_D\Delta\cc\chi\sigma^{2n+1}\rr&=&\int_S\left\langle\dg\chi,\nn\right\rangle\perh=
\int_S\left\langle\dg\left(\frac{|\dg\varphi|^2}{2}\right),\nn\right\rangle\perh\\&=&\label{2iny}
\sum_{j\in
I\cc}\int_S\tau_j(\varphi)\frac{\partial\tau_j(\varphi)}{\partial\nn}\perh=
\int_S\left(\frac{\partial\varphi}{\partial\nn}\frac{\partial^2\varphi}{\partial\nn^2}+
\sum_{j\in
I\ss}\tau_j(\varphi)\frac{\partial\tau_j(\varphi)}{\partial\nn}\right)\perh.\end{eqnarray}
Using the identity $\Delta\ss\varphi=\Delta\cc\varphi +
\MS\frac{\partial\varphi}{\partial\nn}-\left\langle{\rm
Hess}\cc(\varphi)\nn,\nn\right\rangle$ (see Lemma \ref{ljjjkl}) and  the fact that
$$\left(\frac{\partial^2}{\partial\nn^2}-\gc_{\nn}\nn\right)(\varphi)=\left\langle{\rm
Hess}\cc(\varphi) \nn,\nn\right\rangle,$$yields
\begin{eqnarray}\label{fterm}\int_S\left(\frac{\partial\varphi}{\partial\nn}\frac{\partial^2\varphi}{\partial\nn^2}\right)
\perh=
\int_S\nn(\varphi)\left(\underbrace{\Delta\cc\varphi}_{=\psi}-\Delta\ss\varphi+\MS\nn(\varphi)+
\underbrace{\gc_{\nn}\nn(\varphi)}_{=\left\langle\gc_{\nn}\nn,\qq\varphi\right\rangle}\right)\perh.\end{eqnarray}The second integrand of the right-hand side of  \eqref{2iny} can be computed by means of a
 frame $\mathcal F$ satisfying Lemma \ref{Sple}
at a fixed point $p\in S$. More precisely, we have
\begin{eqnarray*}A&:=&\sum_{j\in
I\ss}\int_S\tau_j(\varphi)\left(\nn(\tau_j(\varphi))\right)\perh\\&=&\sum_{j\in
I\ss}\int_S\tau_j(\varphi)\left(\tau_j(\nn(\varphi))+\left(\gc_{\nn}\tau_j-\gc_{\tau_j}\nn\right)(\varphi)\right)\perh\\&=&
\sum_{j\in
I\ss}\int_S\tau_j(\varphi)\left\{\tau_j(\nn(\varphi))+\sum_{l\in
I\cc}\left(\underbrace{\left\langle\gc_{\nn}\tau_j,\tau_l\right\rangle}_{\neq
0\Leftrightarrow
l=1}\tau_l(\varphi)-\left\langle\gc_{\tau_j}\nn,\tau_l\right\rangle\tau_l(\varphi)\right)\right\}\perh\\&=&\sum_{j\in
I\ss}\int_S\tau_j(\varphi)\left\{\tau_j(\nn(\varphi))+\sum_{l\in
I\ss}\left(-\left\langle\gc_{\nn}\nn,\tau_j\right\rangle\,\nn(\varphi)+
B\cc(\tau_j,\tau_l)\tau_l(\varphi)\right)\right\}\perh\\&=&
\int_S\left\{\left\langle\qq\varphi,\qq\left(\nn(\varphi)\right)\right\rangle-\left\langle\gc_{\nn}\nn,\qq\varphi\right\rangle\,\nn(\varphi)+
B\cc\left(\qq\varphi,\qq\varphi\right)\right\}\perh.\end{eqnarray*}  Theorem \ref{GD2} allows us to integrate by parts the first integrand and, since
the
boundary term vanishes, we get that
$$\int_S\lg\left(\nn(\varphi)\qq\varphi\right)\perh=
\int_S\left(\div\ss\left(\nn(\varphi)\qq\varphi\right)-\varpi\left\langle\nn\op,\left(\nn(\varphi)\qq\varphi\right)\right\rangle\right)\perh=0$$
and hence$$\int_S\left\langle\qq\varphi,\qq(\nn(\varphi))\right\rangle\perh=
-\int_S\left(\nn(\varphi)\Delta\ss\varphi-\varpi\nn(\varphi)\nn\op(\varphi)\right)\perh.$$
Therefore
\begin{eqnarray}\label{vucinic}A=\int_S\left(-\nn(\varphi)\left(\Delta\ss\varphi-\varpi\nn\op(\varphi)+\left\langle\gc_{\nn}\nn,\qq\varphi\right\rangle\right)+
B\cc\left(\qq\varphi,\qq\varphi\right)\right)\perh.\end{eqnarray}  Finally, by
making use of \eqref{aquilani}, \eqref{fterm} and \eqref{vucinic} we  obtain
\begin{eqnarray*}
&&\int_S\left\{-2\nn(\varphi)\left(\Delta\ss\varphi-\frac{\varpi\nn\op(\varphi)}{2}\right)+\MS(\nn(\varphi))^2+
S\cc\left(\qq\varphi,\qq\varphi\right) \right\}\perh \\&=& \int_D \left(\|{\rm
Hess}\cc\phi\|^2\ngr-\psi^2+2\left\langle T(\dg\phi)
,C\cc^{2n+1}\dg\phi\right\rangle\right)\sigma^{2n+1}\rr,\end{eqnarray*}which
is equivalent to the thesis, once we note that
$$\lh\varphi+\frac{\varpi\nn\op(\varphi)}{2}=\Delta\ss\varphi-\frac{\varpi\nn\op(\varphi)}{2}.$$This achieves the proof.
\end{proof}
\section{Some applications}\label{SECTION4}
Let us begin with the following:
\begin{oss}Let $\phi\in\cont^2(\mathbb H^n)$. In general, the horizontal Hessian  ${\rm Hess}\cc\phi$ of $\phi$ it {\rm is not symmetric}.
However, we may consider its standard
decomposition $${\rm Hess}\cc\phi=\frac{{\rm Hess}\cc\phi+{\rm
Hess}^{\rm Tr}\cc\phi}{2}+\frac{{\rm Hess}\cc\phi-{\rm
Hess}\cc^{\rm Tr}\phi}{2}:={\rm Hess}\cc^{\rm sym}\phi+{\rm
Hess}\cc^{\rm skew}\phi$$ where ${\rm Hess}\cc^{\rm sym}\phi$
denotes the symmetric part of  ${\rm Hess}\cc\phi$ and ${\rm
Hess}\cc^{\rm skew}\phi$ denotes its skew-symmetric part. It is
not  difficult to see that $${\rm Hess}\cc^{\rm
skew}\phi=-\frac{T\phi}{2}C\cc^{2n+1}.$$ Indeed note that, by its own definition, the $i$-th row of ${\rm Hess}\cc^{\rm
skew}\phi$ is given by
$\frac{\dg(X_i\phi)-X_i(\dg\phi)}{2}$, and  one has
$[X_j,X_i](\phi)=(X_jX_i-X_iX_j)(\phi)=\SC^{2n+1}_{j\,i} T\phi$.
Therefore$$\|{\rm Hess}\cc\phi\|^2\ngr=\|{\rm Hess}^{\rm
sym}\cc\phi\|^2\ngr+ \frac{n}{2}(T\phi)^2.$$The last
identity just says that  the Gram norm of a matrix is the sum of
the Gram norm of its symmetric part with the Gram norm of its
skew-symmetric part.
Finally, it is elementary to see that $\|C^{2n+1}\cc\|^2\ngr=2n$.\end{oss}

By applying  Theorem \ref{RELL} together with
Newton's inequality we deduce an interesting inequality.
\begin{corollario}\label{c0f}Under the same hypotheses of Theorem \ref{RELL},
the following holds:
\begin{eqnarray}\label{erllo}&&\frac{2n-1}{2n}\int_D \psi^2\sigma^{2n+1}\rr\geq \int_D \left(
\frac{n}{2}(T\phi)^2-2\left\langle\dg \left(T\phi\right)
,(\dg\phi)\op\right\rangle\right)\sigma^{2n+1}\rr\\\nonumber&+&
\int_S\left(2\frac{\partial\varphi}{\partial\nn}\left(\lh\varphi-\frac{\varpi}{2}\frac{\partial\varphi}{\partial\nn\op}\right)-\MS
\left(\frac{\partial\varphi}{\partial\nn}\right)^2-
S\cc\left(\qq\varphi,\qq\varphi\right) \right)\perh,\end{eqnarray}with
equality if, and only if, ${\rm Hess}\cc^{\rm
sym}\phi=\frac{\psi}{2n}{\rm Id}\cc$.\end{corollario}

Note that ${\rm Id}\cc\cong {\bf 1}_{2n\times 2n}\in\mathcal M_{2n\times 2n}(\R)$.

\begin{proof}Using Newton's inequality yields
\begin{eqnarray*}\|{\rm Hess}\cc\phi\|^2\ngr&=&\|{\rm Hess}^{\rm
sym}\cc\phi\|^2\ngr+ \frac{n}{2}(T\phi)^2\\&\geq& \frac{{\rm
Tr}^2({\rm Hess}\cc^{\rm
sym}\phi)}{2n}+\frac{n}{2}(T\phi)^2\\&=& \frac{\Delta\cc\phi}{2n}+\frac{n}{2}(T\phi)^2\\
&=&\frac{\psi^2}{2n}+\frac{n}{2}(T\phi)^2.\end{eqnarray*}As it
is well-known, one has equality in this inequality if, and only
if, ${\rm Hess}\cc^{\rm sym}\phi=\frac{\psi}{2n}{\rm Id}\cc$.
From this argument and Theorem \ref{RELL} we easily get
\eqref{erllo} and the thesis follows.\end{proof}

The next three corollaries will follow from Theorem \ref{RELL} by making appropriate choices of the \textquotedblleft test function\textquotedblright $\varphi:D\longrightarrow\R$.

\begin{corollario}\label{c1f}Let $D\subset\mathbb{H}^n$ and let $S=\partial D$ be a
$\cont^2$-smooth compact (closed)
 hypersurface. Then $$
\int_S\left\{ \MS\left\langle V,\nn\right\rangle^2  -
S\cc\left(V\ss,V\ss\right) \right\}\perh=3\int_S \varpi\left\langle V,\nn\right\rangle\left\langle V,\nn\op\right\rangle\,\perh.$$

\end{corollario}

\begin{proof} Let $V\in\XH$ be a constant left-invariant vector field and take $\varphi=\left\langle V,\x\right\rangle$. Then, we have
\begin{itemize}
 \item $\dg\phi=V$;
\item $T\phi=0$;
\item $\Delta\cc\phi=0$;
\item ${\rm
Hess}\cc\phi=0_{2n\times2n}$;
\item $\frac{\partial\varphi}{\partial\nn}=\left\langle V,\nn\right\rangle$;
\item $\frac{\partial\varphi}{\partial\nn\op}=\left\langle V,\nn\op\right\rangle$;
\item $\qq\varphi=V\ss=V-\left\langle V,\nn\right\rangle\,\nn$;
\item $\Delta\ss\varphi=\div\ss\left(\qq\varphi\right)=\MS\left\langle V,\nn\right\rangle$;
\item $\lh\varphi=\left\langle\left(\MS\nn-\varpi\nn^\perp\right), V\right\rangle=\MS\left\langle V,\nn\right\rangle-\varpi\left\langle V,\nn^\perp\right\rangle$.
\end{itemize}
By substituting the previous calculations into the identity of Theorem \ref{RELL}, we get that the left hand side of the identity vanishes. Therefore, one has
$$
\int_S\left\{2\left\langle V,\nn\right\rangle\left(\MS\left\langle V,\nn\right\rangle-\frac{3}{2}\varpi\left\langle V,\nn\op\right\rangle\right)-\MS
\left\langle V,\nn\right\rangle^2-
S\cc\left(V\ss,V\ss\right) \right\}\perh=0.$$Hence
$$
\int_S\left\{ \MS\left\langle V,\nn\right\rangle^2 -3\varpi\left\langle V,\nn\right\rangle\left\langle V,\nn\op\right\rangle  -
S\cc\left(V\ss,V\ss\right) \right\}\perh=0,$$which is equivalent to the thesis.
\end{proof}

\begin{corollario}\label{c2f}Let $D\subset\mathbb{H}^n$ and let $S=\partial D$ be a
$\cont^2$-smooth compact (closed)
 hypersurface. Then
 $$\Vol(D)= \frac{1}{2n}\left\{\int_S  3 \int_S\varpi\left\langle\x,\nn\right\rangle\left\langle\x,\nn\op\right\rangle\,\per- \int_S\left\{\MS \left\langle\xp,\nn\right\rangle^2 -S\cc\left(\xsp,\xsp\right)\right\}\perh\right\}.$$
\end{corollario}

\begin{proof}Let   $\varphi=2t$. Then, we have
 \begin{itemize}
 \item $\dg\phi=\xp$,
\item $T\phi=2$
\item $\Delta\cc\phi=0$,
\item ${\rm Hess}\cc\phi=-C^{2n+1}\cc$
\item $\qq\varphi=\xsp=\xp-\left\langle \xp,\nn\right\rangle\,\nn$,
\item $\Delta\ss\varphi=\div\ss\left(\qq\varphi\right)=\MS\left\langle \xp,\nn\right\rangle$,
\item $\lh\varphi=\MS\left\langle \xp,\nn\right\rangle-\varpi\left\langle\x,\nn\right\rangle $.\end{itemize}
We also stress that $$ \lh\varphi=-\left\langle\left(\MS\nn\op+\varpi\nn\right), \x\right\rangle=-\left\langle\left(\MS\nn-\varpi\nn\op\right)\op, \x\right\rangle= \left\langle\left(\MS\nn-\varpi\nn\op\right), \xp\right\rangle.$$Hence, using Theorem \ref{RELL} yields\begin{eqnarray*} &&-\int_D \left\|C^{2n+1}\cc\right\|^2\ngr\sigma^{2n+1}\rr =-2n\Vol(D)\\&=&
\int_S\left\{2\left\langle \xp,\nn\right\rangle\left(\MS\left\langle \xp,\nn\right\rangle-\varpi\left\langle\x,\nn\right\rangle -\frac{\varpi}{2}\left\langle\x,\nn\right\rangle\right)-\MS
\left\langle \xp,\nn\right\rangle^2-
S\cc\left(\xsp,\xsp\right) \right\}\perh\\&=&
\int_S\left\{ \MS\left\langle \xp,\nn\right\rangle^2 -3\varpi\left\langle\x,\nn\right\rangle\left\langle \xp,\nn\right\rangle -
S\cc\left(\xsp,\xsp\right) \right\}\perh,\end{eqnarray*}which implies the thesis.
\end{proof}

\begin{corollario}\label{c3f}Let $D\subset\mathbb{H}^n$ and let $S=\partial D$ be a
$\cont^2$-smooth compact (closed)
 hypersurface. Then $$\Vol(D)= \frac{1}{2n(2n-1)}\left\{\int_S  3 \int_S\varpi\left\langle\x,\nn\right\rangle\left\langle\x,\nn\op\right\rangle\,\per- \int_S\left\{\MS \left\langle\x,\nn\right\rangle^2 -S\cc\left(\xs,\xs\right)\right\}\perh\right\}.$$
\end{corollario}

\begin{proof}Let   $\varphi= \frac{\rho^2}{2}$, where $\rho:=\|\x\|\eu=\sqrt{\sum_{i=1}^h x_i^2}$. Then, we have
 \begin{itemize}
 \item $\dg\phi=\x$,
\item $T\phi=0$
\item $\Delta\cc\phi=2n$,\item ${\rm Hess}\cc\phi={\bf 1}_{2n\times 2n}$
\item $\qq\varphi=\xs=\x-\left\langle \x,\nn\right\rangle\,\nn$,
\item $\Delta\ss\varphi=\div\ss\left(\qq\varphi\right)=(h-1)+\MS \left\langle\x,\nn\right\rangle$,
\item $\lh\varphi=(2n-1)+\left\langle\left(\MS\nn-\varpi\nn\op\right),\x\right\rangle=(2n-1)+\MS \left\langle\x,\nn\right\rangle-\varpi\left\langle\x,\nn\op\right\rangle$.\end{itemize}Thus, substituting these computations into the   identity of Theorem \ref{RELL} yields
\begin{eqnarray*}&&\int_D \left((2n)^2-2n\right)\sigma^{2n+1}\rr=2n(2n-1)\Vol(D)\\
&=&\int_S\left\{2 \left\langle\x,\nn\right\rangle\left((2n-1)+\MS \left\langle\x,\nn\right\rangle-\varpi\left\langle\x,\nn\op\right\rangle-\frac{1}{2}\varpi\left\langle\x,\nn\op\right\rangle\right)-\MS\left\langle\x,\nn\right\rangle^2-
S\cc\left(\xs,\xs\right)\right\}\perh\\&=&\int_S\left\{2(2n-1)\left\langle\x,\nn\right\rangle +\MS \left\langle\x,\nn\right\rangle^2 -3\varpi\left\langle\x,\nn\right\rangle\left\langle\x,\nn\op\right\rangle -
S\cc\left(\xs,\xs\right)\right\}\perh.\end{eqnarray*}
Since $\int_D\div\rr\x\,\sigma^{2n+1}=2n\Vol(D)=\int_S\left\langle\x,\nn\right\rangle\,\per$, we get that
$$2n(2n-1)\Vol(D)+\int_S\left( \MS \left\langle\x,\nn\right\rangle^2 -3\varpi\left\langle\x,\nn\right\rangle\left\langle\x,\nn\op\right\rangle -
S\cc\left(\xs,\xs\right)\right)\perh=0.$$This can be rewritten as
$$\Vol(D)= \frac{1}{2n(2n-1)}\int_S\left( -\MS \left\langle\x,\nn\right\rangle^2 +3\varpi\left\langle\x,\nn\right\rangle\left\langle\x,\nn\op\right\rangle +
S\cc\left(\xs,\xs\right)\right)\perh,$$which is equivalent to the thesis.
\end{proof}

Another interesting  formula can be obtained by using jointly both Corollary \ref{c2f} and Corollary  \ref{c3f}.

\begin{corollario}\label{c4f}Let $n>1$, let $D\subset\mathbb{H}^n$ and let $S=\partial D$ be a
$\cont^2$-smooth compact (closed)
 hypersurface. Then
$$\Vol(D)=\frac{1}{4n(n-1)}\int_S\left\{-\MS \left(\left\langle\x,\nn\right\rangle^2-\left\langle\xp,\nn\right\rangle^2\right) +\left[S\cc\left(\xs,\xs\right)-S\cc\left(\xsp,\xsp\right)\right] \right\}\perh.$$
\end{corollario}
\begin{proof}Immediate.\end{proof}

\subsection{An integral formula for the horizontal mean curvature}\label{SECTION4.1}
Let $D\subset \mathbb H^n$ be a smooth, say $\cont^3$,  bounded domain (i.e. open and connected) and assume that there exists a (global) defining function $f:\mathbb H^n\longrightarrow\R$ for $D$. This means that\begin{itemize}
 \item $D=\left\lbrace  x\in \mathbb H^n: f(x)<0\right\rbrace$,
\item $D^c=\mathbb H^n\setminus D=\left\lbrace  x\in \mathbb H^n: f(x)>0\right\rbrace$,
\item $ \grad\,f\neq 0$ at every point $x\in \partial D$.
\end{itemize}
From now on we will set $S:=\partial D$. The outward unit normal along $S$ is given by $\nu=\frac{ \grad\,f}{| \grad\,f|}$.
It will be useful to replace the  defining function $f$ for $D$ with the function $\widetilde{f}=\frac{f}{|\dg f|}$.
In fact, this new function has the remarkable feature that $|\dg \widetilde{f}|=1$ along $S$. However, in general,
the function $\widetilde{f}$ is just of class $\cont^2$ (i.e  is one order of differentiability less smooth than $f$) on $S\setminus C_S$ and fails to be smooth only at $C_S$.

Thus, applying Theorem \ref{RELL} to the function $\phi:=\widetilde{f}$ yields the formula:
\begin{eqnarray}\label{pfa} \int_D \left\{\left(\Delta\cc\phi\right)^2-\|{\rm
Hess}\cc\phi\|^2\ngr+2\left\langle \dg \left(T\phi\right)
,(\dg\phi)\op\right\rangle\right\}\sigma^{2n+1}\rr=
-\int_S\MS\,\perh.\end{eqnarray}Note that we have used $\grad\,\varphi=\nu$ (which implies $\qq\varphi=0$ and $\lh\varphi=0$) where $\varphi=\phi|_S$.

Now let $S$ be a compact $\cont^3$-smooth embedded hypersurface. A similar formula can be obtained when we
 consider a foliation of a small spatial neighborhood of $S$.  More precisely, let  $f:\mathbb H^n\times]-\varepsilon, \varepsilon[\longrightarrow\R$  be a $\cont^3$-smooth function such that:
\begin{itemize}
 \item $S_t=\{x\in\mathbb H^n: f_t(x)=f(x, t)=0\,\,\forall\,  t\in]-\varepsilon, \varepsilon[\}$,
\item $| \grad\,f_t|\neq 0$ along $S_t$ for every $t\in]-\varepsilon, \varepsilon[$,
\item $|\dg f_t|=1$ at each NC point of $S_t$.
\end{itemize}
Moreover, let $D:=\left\lbrace x\in\mathbb H^n: f_t(x)\in]-\varepsilon, \varepsilon[\right\rbrace$ and set $S^\pm=\left\lbrace x\in\mathbb H^n: f_t(x)=\pm\varepsilon\right\rbrace $.

We again apply Theorem \ref{RELL} to the function $\phi:=f$ and we similarly get
\begin{eqnarray}\label{pfa2} \int_D \left\{\left(\Delta\cc\phi\right)^2-\|{\rm
Hess}\cc\phi\|^2\ngr+2\left\langle \dg \left(T\phi\right)
,(\dg\phi)\op\right\rangle\right\}\sigma^{2n+1}\rr=
-\int_{S^+\cup S^-}\MS\,\perh,\end{eqnarray}where we have used $\grad\,\varphi=\nu$ (so that $\qq\varphi=0$ and $\lh\varphi=0$) and $\varphi=\phi|_S$.
The previous assumptions allow us to say something more. But before this, we need
 the following  corollary of the classical  Coarea formula:
\begin{Prop} Let
$D\subset\mathbb H^n$ be a smooth domain and let
$\phi\in\cont^1(D)$. Then
\begin{eqnarray} \int_{D}\psi|\dg\phi(x)|\,\sigma^{2n+1}\rr(x)=\int_{\R}ds\left( \int_{\phi^{-1}[s]\cap
D}\psi\,\perh\right).\end{eqnarray}\end{Prop}

Applying this formula yields \begin{eqnarray*}&&\int_D \left\{\left(\Delta\cc\phi\right)^2-\|{\rm
Hess}\cc\phi\|^2\ngr+2\left\langle \dg \left(T\phi\right)
,(\dg\phi)\op\right\rangle \right\}\sigma^{2n+1}\rr\\&=&\int_{-\varepsilon}^{\varepsilon}ds \left\{\int_{S_t=\phi^{-1}[s]\cap
D}\left(\MST^2-\left\|\mathcal J\cc\nt\right\|^2\ngr+ 2 \frac{\partial \varpi^t}{\partial \nn^{t\,\perp}}\right)\perh\right\},\end{eqnarray*}
where $\nt$ is the unit $\HH$-normal along $S_t$ and $\MST$ denotes the $\HH$-mean curvature of $S_t$.
Furthermore
\begin{eqnarray*}\left\|\mathcal J\cc\nt\right\|^2\ngr&=&\left\|B^t\cc+\gc_{\nt}\nt\right\|^2\ngr=\left\|S^t\cc\right\|^2\ngr+\left\|A^t\cc\right\|^2\ngr+\left\|\gc_{\nt}\nt\right\|^2\ngr\\&=&\left\|S^t\cc\right\|^2\ngr+ \frac{n-1}{2}(\varpi^t)^2+(\varpi^t)^2\left\|C\cc^{2n+1}\nt\right\|^2\ngr\\&=&\left\|S^t\cc\right\|^2\ngr+ \frac{n+1}{2}(\varpi^t)^2,\end{eqnarray*}where we have used  $B\cc^t=S^t\cc+A^t\cc$ together with the identity $\gc_{\nt}\nt=-\varpi^tC\cc^{2n+1}\nt$; see, for instance, \cite{Montestab}.
Hence
\begin{eqnarray}\label{mkl}\int_{-\varepsilon}^{\varepsilon}ds \underbrace{\int_{S_t}\left(\MST^2-\left\|S\cc^t\right\|^2\ngr+
 2 \frac{\partial \varpi^t}{\partial \nn^{t\,\perp}}-\frac{n+1}{2}(\varpi^t)^2\right)\per}_{=II_{S_t}(\nt,\perh)}=
-\int_{S^+\cup S^-}\MS\,\perh,
\end{eqnarray}where $II_{S_t}(\nt,\perh)$ is nothing but the second variation formula of the $\HH$-perimeter $(\perh)_t$ of  $S_t$ for a variation $\vartheta$  having variation vector $W_t=\frac{d}{dt}\vartheta=\nt$; see  \cite{Monteb, Montestab}.

At this point, we may apply another integral formula to each integral over $S_t$. We stress that we are assuming that each $S_t$ is a compact closed hypersurface, at least of class $\cont^2$.

\begin{lemma}\label{mio}Let  $S\subset\mathbb{H}^n$ be a $\cont^2$-smooth compact
hypersurface without boundary. Then
\begin{eqnarray}\label{id2}
 \int_S\left(\frac{\partial\varpi}{\partial\nn^{\perp}}-n\varpi^2\right)\,\perh=0,\end{eqnarray}
whenever $\varpi\nn^\perp$ is admissible  (for the horizontal divergence formula).\end{lemma}
\begin{proof} We have
\begin{eqnarray*}\int_S\lg(\varpi\nn^\perp)\,\perh=\int_S\left(\div\ss(\varpi\nn^\perp) -
\varpi\langle \nn^\perp, \nn^\perp\rangle\right)\,\perh=\int_S\left(\div\ss(\varpi\nn^\perp) -
\varpi^2\right)\,\perh=0.\end{eqnarray*}Since
\begin{eqnarray*} \div\ss( \varpi\nn^{\perp})&=&
  \frac{\partial\varpi}{\partial\nn^{\perp}}+\varpi \div\ss(\nn^{\perp})
\\&=&
 \frac{\partial\varpi}{\partial\nn^{\perp}}-\varpi \,
\mathrm{Tr}\big(B\cc(\,\cdot\,, C^{2n+1}\ss\,
 \cdot)\big),\end{eqnarray*}where $C^{2n+1}\ss=C^{2n+1}\cc|_{\HS}$, we get that  \begin{eqnarray}\label{id1}
 \int_S\left(\frac{\partial\varpi}{\partial\nn^{\perp}}-\varpi^2\right)\,\perh=
 \int_S\left(\mathrm{Tr}\big(B\cc(\,\cdot\,, C^{2n+1}\ss\,
 \cdot)\big)\right)  \varpi \perh.\end{eqnarray}
Since
 $\mathrm{Tr}\big(B\cc(\,\cdot\,, C^{2n+1}\ss\,
 \cdot)\big)=(n-1)\varpi$,  the thesis follows; see \cite{Monteb, Montestab}.\end{proof}

Finally, by using \eqref{mkl} and the last lemma,
 we have proved the following:

\begin{corollario}\label{Corollary4.9}Under the previous assumptions, the following holds:
$$\int_{-\varepsilon}^{\varepsilon}ds   \int_{S_t}\left(\MST^2-\left\|S\cc^t\right\|^2\ngr+   \frac{3n-1}{2}(\varpi^t)^2\right)\perh =
-\int_{S^+\cup S^-}\MS\,\perh.$$
\end{corollario}
\begin{proof}Immediate.\end{proof}

\rm

{\footnotesize \noindent Francescopaolo Montefalcone:
\\Dipartimento di Matematica Pura e Applicata\\Universit\`a degli Studi di Padova,\\
  Via Trieste, 63, 35121 Padova (Italy)\,
 \\ {\it E-mail address}:  {\textsf montefal@math.unipd.it}}

\end{document}